\documentclass[reqno]{amsart}
\usepackage[T1]{fontenc}
\usepackage[latin1]{inputenc}
\usepackage{amssymb}
\usepackage{amsfonts}
\usepackage{amsmath,mathtools}
\usepackage{mathrsfs}
\usepackage{graphicx}
\usepackage{color}
\usepackage{esint}
\usepackage{verbatim}
\usepackage{bbm}
\usepackage{tikz}
\usepackage{setspace}
\usepackage{stackrel}
\usepackage{academicons}
\usepackage{xcolor}
\usepackage{svg}

 \makeatletter
 \@namedef{subjclassname@2020}{%
 \textup{2020} Mathematics Subject Classification}
 \makeatother

\newcommand{\N}[1]{\mathbb{N}^{#1}}
\newcommand{\R}[1]{\mathbb{R}^{#1}}
\renewcommand{\S}[1]{\mathbb{S}^{#1}}

\newcommand{\cA}{\mathcal A}

\newcommand{\cC}{\mathcal C}

\newcommand{\cH}{\mathcal H}

\newcommand{\cL}{\mathcal L}
\newcommand{\cM}{\mathcal M}

\newcommand{\cR}{\mathcal R}

\newcommand{\ED}{\mathcal{ED}}

\newcommand{\bB}{\mathbf B}

\newcommand{\bulk}{\mathrm{bulk}}

\newcommand{\de}{\mathrm d}

\newcommand{\seq}{\mathrm{seq}}

\newcommand{\surface}{\mathrm{surf}}

\newcommand{\eps}{\varepsilon}

\newcommand{\norma}[1]{\left\lVert#1\right\rVert}

\newcommand{\nc}{\mathrm{nc}}

\newcommand{\tr}{\operatorname{tr}}

\renewcommand{\geq}{\geqslant}
\renewcommand{\leq}{\leqslant}

\newcommand{\longrightharpoonup}{\relbar\joinrel\rightharpoonup}
\newcommand{\wto}{\rightharpoonup}
\newcommand{\lwto}{\longrightharpoonup}
\newcommand{\wsto}{\stackrel{*}{\rightharpoonup}}

\newcommand{\wSD}[1]{\stackrel[SD]{#1}\lwto}

\newcommand{\wHSDKp}{\stackrel[HSD_L]{p}\longrightharpoonup}
\newcommand{\wHSDKstar}{\stackrel[HSD_L]{*}\longrightharpoonup}

\newcommand{\average}{{\mathchoice {\kern1ex\vcenter{\hrule
height.4pt width 8pt depth0pt}
\kern-11pt} {\kern1ex\vcenter{\hrule height.4pt width 4.3pt
depth0pt} \kern-7pt} {} {} }}

\newcommand{\res}{\mathop{\hbox{\vrule height 7pt width .5pt depth
0pt\vrule height .5pt width 6pt depth0pt}}\nolimits}

\mathchardef\emptyset="001F

\providecommand{\U}[1]{\protect\rule{.1in}{.1in}}
\numberwithin{equation}{section}
\newtheorem{definition}{Definition}[section]
\newtheorem{theorem}[definition]{Theorem}
\newtheorem{lemma}[definition]{Lemma}
\newtheorem{proposition}[definition]{Proposition}

\theoremstyle{definition} {\newtheorem{remark}[definition]{Remark}}

\makeindex

\title[The variational modeling of hierarchical structured deformations]{The variational modeling of hierarchical structured deformations}%
\author[A.~C.~Barroso]{Ana Cristina Barroso}
\address[A.~C.~Barroso]{Departamento de Matem\'atica and CMAFcIO, 
Faculdade de Ci\^encias da Universidade de Lisboa,
Campo Grande, Edif\' \i cio C6, Piso 1,
1749-016 Lisboa, Portugal}
\email{acbarroso@ciencias.ulisboa.pt}
\author[J.~Matias]{Jos\'{e} Matias}
\address[J.~Matias]{Departamento de Matem\'atica, Instituto Superior T\'ecnico, Av.~Rovisco Pais, 1, 1049-001 Lisboa, Portugal}
\email{jose.c.matias@tecnico.ulisboa.pt}
\author[M.~Morandotti]{Marco Morandotti}
\address[M.~Morandotti]{Dipartimento di Scienze Matematiche ``G.~L.~Lagrange'', Politecnico di Torino, Corso Duca degli Abruzzi, 24, 10129 Torino, Italy}
\email{marco.morandotti@polito.it}
\author[D.~R.~Owen]{David R.~Owen}
\address[D.~R.~Owen]{Department of Mathematical Sciences, Carnegie Mellon University, 5000 Forbes Ave., Pittsburgh, PA 15213 USA}
\email{do04@andrew.cmu.edu}
\author[E.~Zappale]{Elvira Zappale}
\address[E.~Zappale]{Dipartimento di Scienze di Base e Applicate per l'Ingegneria, Sapienza Universit\`{a} di Roma, Via Antonio Scarpa, 16, 00161 Roma, Italy}
\email{elvira.zappale@uniroma1.it}

\dedicatory{In memory of Jerry Ericksen, whose broad and deep scientific contributions and leadership never cease to evoke admiration and to provide inspiration.}

\date{\today}


\makeindex

\begin{document}

\begin{abstract}
Hierarchical (first-order) structured deformations are studied from the variational point of view. 
The main contributions of the present research are the first steps, at the theoretical level, to establish a variational framework to minimize mechanically relevant energies defined on hierarchical structured deformations. 
Two results are obtained here: (i) an approximation theorem and (ii) the assignment of an energy to a hierarchical structured deformation by means of an iterative procedure. This has the effect of validating the proposal made in \cite{DO2019} to study deformations admitting slips and separations at multiple submacroscopic levels.
An explicit example is provided to illustrate the behavior of the proposed iterative procedure and relevant directions for future research are highlighted.
\end{abstract}

\maketitle

\tableofcontents


\section{Introduction}

Refinements of classical continuum theories of elastic bodies have the potential to broaden their range of applicability by capturing phenomena of particular interest or to adapt the theories to specific physical contexts.
Structured deformations \cite{DPO1993} provide a mathematical framework to capture the effects at the macroscopic level of geometrical changes, such as slips and separations, at submacroscopic levels. 
The variational formulation of the theory of structured deformations \cite{CF1997} set the basis for the enrichment of the classes of energies and force systems of interest in variational and field-theoretic descriptions of deformations of elastic bodies, making it unnecessary to commit at the outset to any prototypical mechanical theories, such as elasticity, plasticity, or fracture.

Following \cite{CF1997}, a (first-order) structured deformation is a pair $(g,G)\in SBV(\Omega;\R{d})\times L^1(\Omega;\R{d\times N})\eqqcolon SD(\Omega)$, where $g\colon\Omega\to\R{d}$ is the macroscopic deformation of the body and $G\colon\Omega\to\R{d\times N}$ is a tensor field associated with submacroscopic deformation.
In classical theory of mechanics, the field~$g$ and its gradient~$\nabla g$ alone characterize the deformation of the body; in the framework of structured deformations, the additional geometrical field~$G$ captures the contribution at the macroscopic scale of smooth submacroscopic changes, and the difference $\nabla g-G$ captures the contribution at the macroscopic scale of non-smooth submacroscopic changes, such as slips and separations, which are referred to as \emph{disarrangements} \cite{DO2002}. Accordingly,~$G$ is called the deformation without disarrangements, and, heuristically, the disarrangement tensor $M\coloneqq \nabla g-G$ is an indication of how non classical a structured deformation is: if~$g$ is a Sobolev field and $M=0$, then the deformation without disarrangements $G=\nabla g$ is simply the classical deformation gradient; if $M\neq0$, the effect of submacroscopic slips and separations, which are phenomena involving interfaces, is captured at the macroscopic level.
This fact will be made precise in the Approximation Theorem~\ref{appTHM}.

The classical variational methods for continuum mechanics rely on energy minimization. In \cite{CF1997}, the energy assigned to a structured deformation $(g,G)\in SD(\Omega)$ is the one arising from the most economical way to reach $(g,G)$ by means of a sequence $\{u_n\}\subset SBV(\Omega;\R{d})$ approximating $(g,G)$ in the following sense
\begin{equation}\label{101}
u_n\to g\quad\text{in $L^1(\Omega;\mathbb R^d)$}\qquad\text{and}\qquad \nabla u_n\wsto G\quad\text{in $\mathcal M(\Omega;\mathbb R^{d\times N})$,}
\end{equation}
where $\mathcal M(\Omega;\mathbb R^{d\times N})$ is the set of bounded matrix-valued Radon measures on $\Omega$; the convergence in \eqref{101} will be denoted by $u_n\wSD{*}(g,G)$.

We let the initial energy of a deformation $u\in SBV(\Omega;\mathbb R^d)$ be
\begin{equation}\label{103}
	E(u)\coloneqq \int_\Omega W(x,\nabla u(x))\,\de x+\int_{\Omega\cap S_u} \psi(x,[u](x),\nu_u(x))\,\de\cH^{N-1}(x),
\end{equation}
which is determined by the bulk and surface energy densities $W\colon\Omega\times\R{d\times N}\to[0,+\infty)$ and $\psi\colon\Omega\times\R{d}\times\S{N-1}\to[0,+\infty)$.  In formula~\eqref{103}, $\de x$ and $\de\mathcal H^{N-1}(x)$ denote the $N$-dimensional Lebesgue and $(N-1)$-dimensional Hausdorff measures, respectively; $[u](x)$ and $\nu_u(x)$ denote the jump of~$u$ and the normal to the jump set for each $x\in S_u$, the jump set of~$u$.

In mathematical terms, the process just described to assign an energy to a structured deformation $(g,G)\in SD(\Omega)$ reads 
\begin{equation}\label{102}
I(g,G)\coloneqq \inf\Big\{\liminf_{n\to\infty} E(u_n): u_n \in SBV(\Omega;\R{d}), u_n\wSD{*}(g,G)\Big\}.
\end{equation}
In the language of calculus of variations, the operation in \eqref{102} is known as \emph{relaxation}, and an important goal is to prove that the functional $I$ admits an integral representation, that is, that there exist functions $H\colon\Omega\times\R{d\times N}\times\R{d\times N}\to[0,+\infty)$ and $h\colon\Omega\times \R{d}\times\S{N-1}\to[0,+\infty)$ such that
\begin{equation}\label{104}
I(g,G)=\int_\Omega H(x,\nabla g(x),G(x))\,\de x+\int_{\Omega\cap S_g} h(x,[g](x),\nu_g(x))\,\de\cH^{N-1}(x).
\end{equation}
In the case where the initial bulk and surface energy densities $W$ and $\psi$ do not depend explicitly on the spatial variable~$x$, this was the main result in \cite{CF1997}, whose extension to include the explicit~$x$ dependence can be found in \cite[Theorem~5.1]{MMOZ}.

\smallskip

In the recent contribution \cite{DO2019}, Deseri and Owen extended the theory of \cite{DPO1993} to \emph{hierarchies} of structured deformations in order to include the effects of disarrangements at more than one submacroscopic level. 
This extension is based on the fact that many natural and man-made materials exhibit different levels of disarrangements. 
Muscles, cartilage, bone, plants, and biomedical materials are just some of the materials whose mechanical behavior can be addressed within the generalized field theory proposed in \cite{DO2019}. 
Since a structured deformation $(g,G)$ identifies two levels, the macroscopic one and the submacroscopic one, for $L\in\N{}$, we call an \emph{$(L+1)$-level (first-order) structured deformation} an $(L+1)$-tuple $(g,G_1,\ldots,G_L)$, where $g\colon\Omega\to\R{d}$ is the macroscopic deformation and $G_\ell\colon\Omega\to\R{d\times N}$, for $\ell=1,\ldots,L$, are the deformations without disarrangements at each of the~$L$ submacroscopic levels.

With the far-reaching goal of establishing a variational theory \emph{\`{a} la} Choksi and Fonseca \cite{CF1997} for hierarchies of structured deformations, the main results of the present research are: (i) an approximation theorem for $(L+1)$-level structured deformations (see Theorem~\ref{appTHMh} below) and (ii) a proposal for the assignment of an energy to the $(L+1)$-level structured deformation $(g,G_1,\ldots,G_L)$ by means of a well-posed recursive relaxation process (see Theorem~\ref{thm_RR} and the example in Section~\ref{sec_EX} below).

\smallskip

After stating the notation that we use throughout this work, the plan of the paper is the following: in Section~\ref{sec_SD} we collect some results coming from the theory of structured deformations, together with some recent generalizations; we also state and prove Theorem~\ref{thm_propdens} on the properties of the relaxed bulk and surface energy densities: this is the crucial result that allows the recursive relaxation process to continue through all levels.
In Section~\ref{sec_hierarchies}, we present our main results on hierarchies of structured deformations, namely the Approximation Theorem~\ref{appTHMh} and Theorem~\ref{thm_RR} on the well posedness of the recursive relaxation and, therefore, on the assurance that the energy assigned to an $(L+1)$-level structured deformation $(g,G_{1},\cdots ,G_{L})$ is
well-defined.
We also show an example in which explicit computations can be carried out.
Finally, Section~\ref{sec_CO} offers an outlook for future directions of research.

\subsection{Notation} 
We will use the following notations
\begin{itemize}
	\item $\mathbb N$ denotes the set of natural numbers without the zero element;
	\item $\Omega \subset \mathbb R^{N}$ is a bounded connected open set; 
	\item $\mathbb S^{N-1}$ denotes the unit sphere in $\mathbb R^N$;
	\item $Q\coloneqq (-\tfrac12,\tfrac12)^N$ denotes the open unit cube of $\mathbb R^{N}$ centred at the origin; for any $\nu\in\mathbb S^{N-1}$, $Q_\nu$ denotes any open unit cube in $\mathbb R^{N}$ with two faces orthogonal to~$\nu$; 
	\item ${\mathcal A}(\Omega)$ is the family of all open subsets of $\Omega $; 
	\item $\mathcal L^{N}$ and $\mathcal H^{N-1}$ denote the  $N$-dimensional Lebesgue measure and the $\left(  N-1\right)$-dimensional Hausdorff measure in $\mathbb R^N$, respectively; the symbol $\de x$ will also be used to denote integration with respect to $\mathcal L^{N}$; 
	\item $\mathcal M(\Omega;\mathbb R^{d\times N})$ is the set of finite matrix-valued Radon measures on $\Omega$; $\mathcal M ^+(\Omega)$ is the set of non-negative finite Radon measures on $\Omega$;
	given $\mu\in\mathcal M(\Omega;\mathbb R^{d\times N})$, 
	the measure $|\mu|\in\mathcal M^+(\Omega)$ 
	denotes the total variation of $\mu$;
	\item $SBV(\Omega;\mathbb R^d)$ is the set of vector-valued \emph{special functions of bounded variation} defined on $\Omega$. 
	Given $u\in SBV(\Omega;\mathbb R^d)$, its distributional gradient $Du$ admits the decomposition $Du=D^au+D^su=\nabla u\cL^N+[u]\otimes\nu_u\mathcal H^{N-1}\res S_u$, where $S_u$ is the jump set of~$u$, $[u]$ denotes the jump of~$u$ on $S_u$, and $\nu_u$ is the unit normal vector to $S_u$; finally, $\otimes$ denotes the dyadic product; 
	\item $L^p(\Omega;\mathbb R^{d\times N})$ is the set of matrix-valued $p$-integrable functions; 
	\item for $p\geq1$, $SD^p(\Omega)\coloneqq SBV(\Omega;\mathbb R^d)\times L^p(\Omega;\mathbb R^{d\times N})$ is the space of structured deformations $(g,G)$ (notice that $SD^1(\Omega)$ is the space $SD(\Omega)$ introduced in \cite{CF1997}); the norm in $SD(\Omega)$ is defined by $\norma{(g,G)}_{SD(\Omega)}\coloneqq\norma{g}_{BV(\Omega;\R{d})}+\norma{G}_{L^1(\Omega;\R{d\times N})}$;
	\item  $C$ represents a generic positive constant that may change from line to line.
\end{itemize}

\section{Preliminaries on structured deformations}\label{sec_SD}
In this section, we present the two results of great importance in the theory of structured deformations, namely the Approximation Theorem (see Theorem~\ref{appTHM} below) and the integral representation result of Theorem~\ref{thm_CF} generalizing that in \cite[Theorems~2.16 and~2.17]{CF1997}, see also \cite[Theorem~5.1]{MMOZ}. Since the energy for a structured deformation $(g,G)$ is defined by means of a relaxation process, it is necessary to show that the set on which the relaxation is performed is non empty. At the same time, starting from an initial energy~$E$ of integral type, Theorem~\ref{thm_CF} guarantees that also the relaxed energy $I$ is of integral type, and identifies the relaxed bulk and surface energy densities $H$ and $h$, respectively. 
The most significant result of Section~\ref{sec_SD}, Theorem~\ref{thm_propdens}, shows that the relaxed energy densities~$H$ and~$h$, themselves, can serve as initial energy densities for the relaxation process, thus allowing the iterative process in Section~\ref{sec_hierarchies} to provide a well-defined energy for $(L+1)$-level structured deformations.


\subsection{Approximation and definition of the energy}
A fundamental result in the theory of structured deformations is the Approximation Theorem \cite[Theorem~5.8]{DPO1993}, a counterpart of which was recovered in \cite[Theorem~2.12]{CF1997} in the $SBV$ framework and in \cite{S2015} in a broader framework. 
In simple terms, \cite[Theorem~2.12]{CF1997} states that given a structured deformation $(g,G) \in SBV(\Omega; \R{d}) \times L^1(\Omega; \R{d\times N})$, there exists a sequence $\{u_n\} \subset SBV(\Omega; \R{d})$ that approximates it, in the sense that the sequence of functions $n\mapsto u_n$ tends to the field~$g$, and the sequence $n\mapsto \nabla u_{n}$ of the absolutely continuous parts of their gradients tends to the (matrix-valued) field $G$, in suitable senses.
The proof of the approximation theorem rests on the following two results. 
\begin{theorem}[{\cite[Theorem~3]{AL}}]\label{Al}
Let $f \in L^1(\Omega; \R{d{\times} N})$. 
Then there exist $u \in SBV(\Omega; \R d)$, a Borel function $\beta\colon\Omega\to\R{d{\times} N}$, and a constant $C_N>0$ depending only on $N$ such that
\begin{equation}\label{817}
Du = f \,{\cL}^N + \beta \cH^{N-1}\res S_u, \qquad
\int_{S_u\cap\Omega} |\beta(x)| \, \de \cH^{N-1}(x) \leq C_N \lVert f\rVert_{L^1(\Omega; \R{d {\times} N})}.
\end{equation}
\end{theorem}
\begin{lemma}[{\cite[Lemma~2.9]{CF1997}}]\label{ctap}
Let $u \in BV(\Omega; \R d)$. Then there exist piecewise constant functions $\bar u_n\in SBV(\Omega;\R d)$  such that $\bar u_n \to u$ in $L^1(\Omega; \R d)$ and
\begin{equation}\label{818}
|Du|(\Omega) = \lim_{n\to +\infty}| D\bar u_n|(\Omega) = \lim_{n\to +\infty} \int_{S_{\bar u_n}} |[\bar u_n](x)|\; \de\cH^{N-1}(x).
\end{equation}
\end{lemma}

We are now ready to state, and prove for the reader's convenience, the approximation theorem for structured deformations. Even though its proof can be found in the literature, see, \emph{e.g.}, \cite[Theorem~2.12]{CF1997} or \cite[Proposition~2.1]{MMOZ}, it is useful to show it here as a preparation for the proof of Theorem~\ref{appTHMh} below.
\begin{theorem}[Approximation Theorem]\label{appTHM}
There exists $C>0$ (depending only on the dimension $N$) such that for every $(g,G)\in SD(\Omega)$ there exists a sequence $\{u_n\}\subset SBV(\Omega;\R{d})$ converging to $(g,G)$ according to 
\begin{equation}\label{1000}
u_n \to g\quad\text{in $L^1(\Omega; \R{d})$}\qquad \text{and}\qquad \nabla u_n \wsto G\quad\text{in $\cM(\Omega;\R{d\times N})$}
\end{equation}
and such that, for all $n\in\N{}$,
\begin{equation}\label{appEST}
|D u_n|(\Omega)\leq C \|(g,G)\|_{SD(\Omega)}. 
\end{equation}
In particular, 
this implies that, up to a subsequence, 
\begin{equation}\label{centerline}
D^s u_n\wsto (\nabla g-G)\cL^N+D^s g\qquad \text{in $\cM(\Omega;\R{d\times N})$.}
\end{equation}
\end{theorem}
\begin{proof}
Let $(g,G)\in SD(\Omega)$ and, by applying Theorem~\ref{Al} with $f\coloneqq \nabla g-G$, let $u\in SBV(\Omega;\R{d})$ be such that $\nabla u=\nabla g-G$.
Furthermore, let $\bar u_n\in SBV(\Omega;\R{d})$ be a sequence of piecewise constant functions approximating $u$, as per Lemma~\ref{ctap}.
Then the sequence of functions 
\begin{equation}\label{undef}
u_n\coloneqq g+\bar u_n-u
\end{equation}
is easily seen to approximate $(g,G)$ in the sense of \eqref{1000}. 
In fact, $ u_n \to g$ in $L^1$ and $\nabla u_n(x) = G(x)$ for $\cL^N$-a.e.~$x\in\Omega$.
Invoking the triangle inequality, the inequality in \eqref{817}, and \eqref{818}, we obtain for $C=3(1+C_N)$
\begin{equation}\label{819}
|Du_n|(\Omega)\leq C\|(g,G)\|_{SD(\Omega)}\qquad\text{for all $n\in\N{}$ sufficiently large,}
\end{equation}
so that \eqref{appEST} is proved for a suitable ``tail'' of the sequence $u_n$.
The convergence of $u_n\to g$ in $L^1$ implies that $Du_n$ converges to $Dg$ in the sense of distributions.
The uniform bound \eqref{819} ensures the existence of a (not relabeled) weakly-* converging subsequence such that $Du_n\wsto Dg$ in $\cM(\Omega;\R{d\times N})$, so that, since $\nabla u_n\wsto G$ in $\cM(\Omega;\R{d\times N})$, we have
$$
D^s u_n\wsto (\nabla g-G)\cL^N+D^s g\qquad \text{in $\cM(\Omega;\R{d\times N})$,}
$$
which is \eqref{centerline}. The proof is concluded.
\end{proof}

\begin{remark}\label{cvgp>1}
An inspection of the proof of Theorem \ref{appTHM} (see also \cite[Theorem~2.12]{CF1997} and \cite[Proposition~2.1]{MMOZ}) shows that the sequence $\{u_n\}$ in \eqref{undef} satisfies 
\begin{equation}\label{nablaun=G}
\nabla u_n=G.
\end{equation}
Furthermore, if $p>1$ and either the density $W$ in \eqref{103} satisfies a coercivity condition as in \eqref{pr3} below with $q=p$, or a generic bound of the type $\sup_n\{\lVert\nabla u_n\rVert_{L^p(\Omega;\R{d\times N})}\} \leq C$ is imposed on the sequences defining $I(g,G)$ in \eqref{102}, then the second convergence in \eqref{1000} can be strengthened to read $\nabla u_n \wto G \in L^p(\Omega;\mathbb R^{d \times N})$.
On the other hand, as observed in the Approximation Theorem, the existence of sequences approximating in this latter sense is guaranteed by  \eqref{nablaun=G}.
Hence, in the case $p>1$, for $(g,G) \in SD^p(\Omega)$, we use the notation $u_n \wSD{p} (g, G)$ to signify
\begin{equation}\label{1000Lp}
u_n \to g \quad\text{in $L^1(\Omega; \R{d})$}\qquad\text{and}\qquad \nabla u_n \wto G\quad \text{in $L^p(\Omega;\mathbb R^{d \times N})$.}
\end{equation}
\end{remark}

Let $p\geq1$ and let the initial energy $E\colon SBV(\Omega;\R{d})\to[0,+\infty]$ be defined by \eqref{103},
where $W\colon\Omega\times\R{d\times N}\to[0,+\infty)$ and $\psi\colon\Omega\times\R{d}\times\S{N-1}\to[0,+\infty)$ are the initial bulk and surface energy densities, respectively.
For $(g,G)\in SD^p(\Omega)$,  the energy  $I(g,G)$ is defined by
\begin{equation}\label{910}
I(g,G)\coloneqq \inf\Big\{ \liminf_{n\to\infty} E(u_n):  \{u_n\}\in\cR_p(g,G)\Big\},
\end{equation}
where 
\begin{equation}\label{R_p}
\cR_p(g,G)\coloneqq\Big\{\{u_n\}\subset SBV(\Omega;\R{d}): u_n\wSD{p}(g,G) 
\Big\}.
\end{equation}

\subsection{Generalization of the Choksi--Fonseca scheme}
Here we define the class of initial bulk and surface energy densities $W$ and $\psi$ that can be used to define the initial energy $E$ to be relaxed to structured deformations. Compared to the results in \cite{CF1997}, we add the explicit dependence on the space variable and we prove that the relaxed energy densities still belong to the class of initial energies. 
\begin{definition}[Energy densities]\label{gen_ass}
For $q\geq1$, we denote by $\ED(q)$ the collection of pairs $(W,\psi)$ of bulk and surface energy densities, where $W\colon \Omega\times\R{d\times N} \to [0, +\infty)$ and $\psi\colon \Omega\times\R{d}\times \S{N-1} \to [0, +\infty)$ are continuous functions satisfying the following conditions 
\begin{enumerate}
\item \label{W_cara}
 there exists a continuous function $\omega_W\colon[0,+\infty)\to[0,+\infty)$ with $\omega_W(s)\to 0$ as $s\to0^+$ such that, for every $x_0, x_1\in\Omega$ and $A\in \R{d\times N}$,
\begin{equation}\label{pr1}
|W(x_1,A)-W(x_0,A)|\leq\omega_W(|x_1-x_0|)(1+|A|^{q});
\end{equation}
\item \label{(W1)_p}  ($q$-Lipschitz continuity) there exists $C_W >0$ such that, for all $x\in\Omega$ and $A_1,A_2 \in \R{d\times N}$,
\begin{equation}\label{pr2}
|W(x,A_1) - W(x,A_2)| \leq C_W |A_1 - A_2| \big(1+|A_1|^{q-1}+|A_2|^{q-1}\big);
\end{equation}
\item \label{W3} there exists $A_0 \in \mathbb R^{d \times N}$ such that $W(\cdot, A_0)\in L^\infty(\Omega)$;
\item \label{W4} there exists $c_W>0$ such that, for every $(x,A)\in \Omega \times \mathbb R^{d\times N}$,
\begin{equation}\label{pr3}
c_W |A|^{q}-\frac{1}{c_W}\leq W(x,A);
\end{equation}
\item\label{psi_0} (symmetry) for every $x \in \Omega$, $\lambda \in \R{d}$, and $\nu \in \mathbb S^{N-1}$, 
\begin{equation*}
\psi (x, \lambda, \nu)= \psi (x,-\lambda, -\nu);
\end{equation*}
\item\label{(psi1)} there exist $c_\psi,C_\psi > 0$ such that, for all $x\in\Omega$, $\lambda \in \R{d}$, and $\nu \in \S{N-1}$,
\begin{equation}\label{pr4}
c_\psi|\lambda| \leq \psi(x,\lambda, \nu) \leq C_\psi|\lambda |;
\end{equation}
\item\label{(psi2)} (positive $1$-homogeneity) for all $x\in\Omega$, $\lambda \in \R{d}$, $\nu \in \S{N-1}$, and $t >0$
$$\psi(x,t\lambda, \nu) = t\psi(x, \lambda, \nu);$$
\item \label{(psi3)} (sub-additivity) for all $x\in\Omega$, $\lambda_1,\lambda_2 \in \R{d}$, and $\nu \in \S{N-1}$,
\begin{equation*}
\psi(x, \lambda_1 + \lambda_2, \nu) \leq \psi(x,\lambda_1, \nu) +\psi(x,\lambda_2, \nu);
\end{equation*}
\item \label{(psi4)} there exists a continuous function $\omega_\psi\colon[0,+\infty)\to[0,+\infty)$ with $\omega_\psi(s)\to 0$ as $s\to0^+$ such that, for every $x_0,x_1\in\Omega$, $\lambda \in \R{d}$, and $\nu \in \S{N-1}$,
\begin{equation}\label{pr5}
|\psi(x_1,\lambda,\nu)-\psi(x_0,\lambda,\nu)|\leq\omega_\psi(|x_1-x_0|)|\lambda|.
\end{equation}
\end{enumerate}
We say that $(W,\psi)\in\ED^{\nc}(q)$ if~$W$ satisfies only \eqref{W_cara}--\eqref{W3} while~$\psi$ satisfies \eqref{psi_0}--\eqref{(psi4)}. Clearly, $\ED(q)$ is a smaller class then $\ED^{\nc}(q)$.
For energy densities $W\colon\R{d\times N}\to[0,+\infty)$ and $\psi\colon\R{d}\times\S{N-1}\to[0,+\infty)$ (therefore \emph{not} explicitly depending on the space variable), we define the class $\ED_{CF}(q)$ of pairs $(W,\psi)$ satisfying (the obvious modifications of) \eqref{(W1)_p}, \eqref{W4}, \eqref{psi_0}, \eqref{(psi1)}, \eqref{(psi2)}, and \eqref{(psi3)}; as above, we say that $(W,\psi)\in\ED_{CF}^{\nc}(q)$ if~$W$ satisfies only \eqref{(W1)_p} while~$\psi$ satisfies only \eqref{psi_0}--\eqref{(psi3)}.
\color{blue}
\end{definition}
In this work, relevant values for $q$ will be either $q=p>1$  or $q=1$.

\begin{remark}\label{rem 3.2}
We point out the following facts:
\begin{itemize}
\item[(i)]	as observed in \cite[pages 298 and 305]{BDV1996}, the symmetry property \eqref{psi_0} in Definition~\ref{gen_ass} ensures that $\psi(x, \lambda, \nu)$ can be equivalently seen as a function of $x$ and $\lambda \otimes \nu$, for every $x \in \Omega, \lambda \in \mathbb R^d, \nu \in \mathbb S^{N-1}$.
\item[(ii)]
coercivity conditions \eqref{pr3} and \eqref{pr4} can rule out some physical relevant settings in particular some arising from fracture mechanics (see \cite[Remark~3.3]{CF1997}). Extra bounds on the norms of the admissible sequences can overcome these constraints to prove Theorem~\ref{thm_CF} below, as observed in the proof of \cite[Theorem~2.16]{CF1997}, see also Remark~\ref{rmk_coerc}.
Therefore, we can relax the request in property \eqref{(psi1)} of Definition~\ref{gen_ass} by requiring that
\begin{equation}\label{relaxedpsi1}
0 \leq \psi(x, \lambda, \nu) \leq C|\lambda |.
\end{equation}
On the other hand (as evident also in the proof of \cite[formula (4.22) and in the entire Section 5]{CF1997}), our coercivity assumptions on $W$ (or milder ones of linear type) and $\psi$ are crucial to deduce that densities $H$ and $h$ have the same type of properties as $W$ and $\psi$, and thus it will be possible to apply Theorem~\ref{thm_CF}, in turn, to $H$ and $h$.
\item[(iii)]
Notice that, if $\psi$ does not depend explicitly on~$x$, condition \eqref{relaxedpsi1} needs not be imposed as an assumption, since it follows immediately from the continuity and the positive $1$-homogeneity of $\psi$: indeed, continuity and property \eqref{(psi2)} of Definition~\ref{gen_ass} imply that $\psi(0,\nu)=0$ for every $\nu\in\S{N-1}$ and so, for every $(\lambda,\nu)\in(\R{d}\setminus\{0\})\times\S{N-1}$, we have
$$0\leq \psi(\lambda,\nu)=|\lambda|\psi\Big(\frac{\lambda}{|\lambda|},\nu\Big)\leq C|\lambda|,$$ 
where $C=\max_{\S{d-1}\times\S{N-1}}\psi$, which is bounded by continuity.
\end{itemize}
\end{remark}

For $A,B\in\R{d\times N}$, let 
$
a_A(x)\coloneqq Ax
$ 
be the linear function with (constant) gradient $A$ and let
\begin{equation}\label{T001}
\cC
^{\bulk}(A,B)\coloneqq \bigg\{u\in SBV(Q;\R{d}):u|_{\partial Q}=a_A|_{\partial Q}, \int_Q \nabla u(x)\,\de x=B, |\nabla u|\in L^p(Q) \bigg\}; 
\end{equation}
for $\lambda\in\R{d}$ and $\nu\in\S{N-1}$ let
\begin{equation}\label{T002}
\cC
^\surface(\lambda,\nu)\coloneqq \big\{u\in SBV(Q_\nu;\R{d}): u|_{\partial Q_\nu}=s_{\lambda,0,\nu}|_{\partial Q_\nu}, 
\nabla u(x)=0\;\text{for $\cL^N$-a.e.~$x\in Q_\nu$}\big\},
\end{equation}
where, for $\lambda_1,\lambda_2\in\R{d}$ and $\nu\in\S{N-1}$, the function $s_{\lambda_1,\lambda_2,\nu}$ is defined by
\begin{equation}\label{snu}
s_{\lambda_1,\lambda_2,\nu}(x)\coloneqq 
\begin{cases}
\lambda_1 & \text{if $x\cdot\nu\geq0$,} \\
\lambda_2 & \text{if $x\cdot\nu<0$.}
\end{cases}
\end{equation}
Given $x\in\Omega$, $U\in\cA(\Omega)$, and $u\in SBV(U;\R{d})$, we define the localized bulk and surface energies for the relaxation process
\begin{subequations}\label{Erelax}
\begin{eqnarray}
E_{x}^{\bulk}(u;U) & \!\!\!\!\coloneqq & \!\!\!\! \int_{U} W(x,\nabla u(y))\,\de y 
+\int_{S_u\cap U} \psi(x,[u](y),\nu_u(y))\,\de\cH^{N-1}(y), \label{Erelaxb}\\
E_{x}^\surface(u;U) & \!\!\!\!\coloneqq & \!\!\!\! \int_{S_u\cap U} \psi(x,[u](y),\nu_u(y))\,\de\cH^{N-1}(y), \label{Erelaxs}
\end{eqnarray}
\end{subequations}
and we define the relaxed bulk and surface energy densities $H\colon\Omega\times\R{d\times N}\times\R{d\times N}\to [0,+\infty)$ and $h\colon\Omega\times\R{d}\times\S{N-1}\to[0,+\infty)$, respectively, by
\begin{subequations}\label{rel_en_dens}
\begin{eqnarray}
H
(x,A,B)& \!\!\!\!\coloneqq &\!\!\!\! \inf\big\{E_{x}^{\bulk}(u;Q): 
 u\in\cC
 ^\bulk(A,B)\big\}, \label{906}\\
h
(x,\lambda,\nu) &\!\!\!\!\coloneqq &\!\!\!\! \inf\big\{ 
E_{x}^\surface(u;Q_\nu): u\in\cC
^\surface(\lambda,\nu)\big\}. \label{907}
\end{eqnarray}
\end{subequations}

The first result that we prove is the sequential characterization of the bulk energy density $H$ defined in \eqref{906}.
\begin{proposition}[Sequential characterization of $H$]\label{seq_char} Under Assumptions \ref{gen_ass}-\eqref{W_cara}, \eqref{(W1)_p}, \eqref{(psi1)}, \eqref{(psi3)}, and \eqref{(psi4)}, for every $(x,A,B)\in\Omega\times\R{d\times N}\times\R{d\times N}$, we have $H(x,A,B)=\widetilde H(x,A,B)$, where
\begin{equation}\label{906sc}
\widetilde H(x,A,B)\coloneqq  \inf\Big\{\liminf_{n\to\infty} E_{x}^{\bulk}(u_n;Q): \{u_n\}\in\cC^\bulk_{\seq}(A,B)\Big\}
\end{equation}
and 
\begin{equation}\label{Cbulkseq}
\cC^{\bulk}_{\seq}(A,B)\coloneqq \big\{\{u_n\}\subset SBV(Q;\R{d}): u_n\to a_A \hbox{ in } L^1(Q;\R{d}), \nabla u_n\rightharpoonup B \hbox{ in }L^p(Q;\mathbb R^{d \times N} )\big\}.
\end{equation}
\end{proposition}
\begin{proof}
The proof follows the lines of that of \cite[Proposition~3.1]{CF1997}.
\end{proof}

\begin{theorem}[Integral representation theorem]\label{thm_CF}
Let $p\geq1$ and let $(W,\psi)\in\ED^{\nc}(q)$, with either $q=p$ or $q=1$, as in Definition~\ref{gen_ass}. 
For $(g,G)\in SD^p(\Omega)$, let $I(g,G)$ be defined by \eqref{910}. 
Then the functional $I\colon SD^p(\Omega)\to[0,+\infty)$ admits the integral representation 
\begin{equation}\label{int_rep_I}
I(g,G)=\int_\Omega H(x,\nabla g(x), G(x))\,\de x+\int_{\Omega\cap S_g} h(x,[g](x),\nu_g(x))\,\de \cH^{N-1}(x),
\end{equation}
where $H$ and $h$ are defined in \eqref{rel_en_dens}. 
\end{theorem}
\begin{proof}
This theorem can be found in \cite[Theorem~5.1]{MMOZ}. Here we just highlight some differences from the result in \cite{MMOZ}, which are due to the definition of the class $\cR_p$ in \eqref{R_p}.
We notice that, for $(W,\psi)\in\ED_{CF}^{\nc}(p)$, the result is \cite[Theorem~2.16]{CF1997}, which provides cell formulae for the relaxed bulk and surface energy densities $H$ and $h$, respectively, which are as in \eqref{rel_en_dens}, but with no $x$-dependence.
On the contrary, for $(W,\psi)\in\ED_{CF}^{\nc}(1)$, the result corresponds to\cite[Theorem~2.17]{CF1997}, which provides the cell formula for the relaxed bulk energy density $H$ as in \eqref{906}, and for the relaxed surface energy density $h$ a cell formula similar to that in \eqref{907}, but containing also a contribution of the recession function $W^\infty$ of the initial bulk energy density. This is not the case in our context because the class $\cR_p$ contains a uniform bound on the $L^p$ norms of the gradients of the sequences converging to $(g,G)$, and this is enough to prevent the appearance of $W^\infty$ in the cell formula for $h$.

As done in \cite[Theorem~5.1]{MMOZ}, and previously in \cite[Theorem~3.2]{BMMO2017}, the explicit $x$-dependence is dealt with thanks to conditions \eqref{W_cara} and \eqref{(psi4)} of Definition~\ref{gen_ass}, borrowing techniques from \cite{BBBF1996}.
\end{proof}
\begin{remark}\label{rmk_coerc}
Assuming coercivity of the bulk energy density $W$ (see property \eqref{W4} in Definition~\ref{gen_ass}) would make the proofs a bit easier. Nonetheless, we prefer to state the hypotheses of Theorem~\ref{thm_CF} without this condition for the sake of generality. The integral representation theorem proved in \cite{CF1997} is also stated without the coercivity assumption, which is temporarily added in the proof to obtain $L^p$ boundedness of the gradients of minimizing sequences, and then is removed. In our setting, the very definition of the class $\cR_p$ carries the needed boundedness, so that our Theorem~\ref{thm_CF} can be safely stated without assuming condition \eqref{W4} of Definition~\ref{gen_ass}.
\end{remark}

In the next theorem, we collect the properties of the relaxed energy densities $H$ and $h$ defined in \eqref{rel_en_dens}. We start from a pair $(W,\psi)\in\ED(q)$
and we prove that, via \eqref{rel_en_dens}, they generate a pair $(H^B,h)\in\ED(1)$. To show this, the coercivity of the initial bulk energy density $W$ is not only needed, but will also hold true for the density $H^B(x,A)=H(x,A,B)$.
\begin{theorem}[Properties of the relaxed energy densities]\label{thm_propdens}
Let $p>1$ and let $(W,\psi)\in\ED(p)$. Let $H\colon\Omega\times\R{d\times N}\times\R{d\times N}\to[0,+\infty)$ and $h\colon\Omega\times\R{d}\times\S{N-1}\to[0,+\infty)$ be the functions defined in \eqref{rel_en_dens}. 
Let us define, for $B\in\R{d\times N}$, the function $H^B\colon\Omega\times\R{d\times N}\to[0,+\infty)$ by $(x,A)\mapsto H^B(x,A)\coloneqq H(x,A,B)$.
Then $(H^B,h)\in\ED(1)$.
Moreover, the function $H$ is $p$-Lipschitz continuous in the third component, namely for every $(x,A)\in \Omega \times \mathbb R^{d \times N}$ there exists a constant $C>0$ such that for every $B_1,B_2\in\R{d\times N}$,
\begin{equation}\label{H_B}
|H(x,A,B_1)-H(x,A,B_2)|\leq C|B_1-B_2|(1+|B_1|^{p-1}+|B_2|^{p-1}).
\end{equation}
\end{theorem}
\begin{proof}
We start by proving that there exist constants $\bar c_H,\bar C_H>0$ such that for every $x\in\Omega$ and $A,B\in\R{d\times N}$ there holds
\begin{equation}\label{ineq0}
\bar c_H(|A|+|B|^p)-\frac1{\bar c_H}\leq H(x,A,B)\leq \bar C_H(1 +|A| +|B|^p).
\end{equation}
The sequence $\{u_n\}$ provided by Theorem~\ref{appTHM} to approximate the structured deformation $(a_A,B)$ belongs to $\cC^{\bulk}_{\seq}(A,B)$, so that, by Proposition~\ref{seq_char}, we obtain the upper bound in \eqref{ineq0} by means of the following chain of inequalities 
\[\begin{split}
H(x,A,B)\leq&\, \liminf_{n\to\infty} \bigg(\int_Q W(x,\nabla u_n(y))\,\de y+\int_{S_{u_n}\cap Q} \psi(x,[u_n](y),\nu_{u_n}(y))\,\de\cH^{N-1}(y) \bigg) \\
\leq&\, \liminf_{n\to\infty} \big(W(x,B)+C_\psi |D^s u_n|(Q)\big) \leq \bar C_H(1+|A|+|B|^p),
\end{split}\]
where we have used the fact that $\nabla u_n=B$ and \eqref{pr4} in the second inequality, and \eqref{pr2} jointly with property~\eqref{W3} in the third one; the constant $\bar C_H$ depends on~$N$, on the norm of $A_0$, on $\norma{W(\cdot,A_0)}_{L^\infty(\Omega)}$, and on $C_W$ and $C_\psi$. 

On the other hand, by \eqref{906sc} there exists a recovery sequence converging to $(a_A, B)$ in the sense of \eqref{1000Lp}.
Thus, by the coercivity in \eqref{pr3} and \eqref{pr4} and the lower semicontinuity of the $L^p$ norm and of the total variation of measures, we obtain that, for $\tilde c_H\coloneqq \min\{c_W, c_\psi\}$,
\begin{equation}\label{H coercivity}
H(x,A,B)\geq \tilde c_H(|B|^p+ |A-B|)-\frac{1}{\tilde c_H} \geq \bar c_H(|A|+ |B|^p)- \frac{1}{\bar c_H},
\end{equation} 
for a suitable constant $\bar c_H\in(0,\tilde c_H)$, so that \eqref{ineq0} is proved.

To show that $H^B$ satisfies property \eqref{W_cara} of Definition~\ref{gen_ass} with $q=1$, we argue in the following way. 
Let $x_0,x_1\in\Omega$ and $A\in\R{d\times N}$ be given, and let $\eps>0$. 
By definition \eqref{906}, there exists $u^\eps\in\cC^{\bulk}(A,B)$ such that
\begin{equation}\label{ineq1}
E_{x_1}^{\bulk}(u^\eps; Q)< H(x_1,A,B)+ \eps=H^B(x_1,A)+\eps.
\end{equation} 
Therefore, since the same $u^\eps$ is also a competitor for $H(x_0,A,B)$, we have, defining the function $\tilde\omega_H(s)\coloneqq \max\{\omega_W(s),\omega_\psi(s)\}$,
\[\begin{split}
H^B(x_0,A)-&\, H^B(x_1,A)=H(x_0,A,B)-H(x_1,A,B)\leq E_{x_0}^{\bulk}(u^\eps;Q)-E_{x_1}^{\bulk}(u^\eps;Q)+\eps \\
\leq&\, \omega_W(|x_0-x_1|)\int_Q (1+|\nabla u^\eps(y)|^p)\,\de y + \omega_\psi(|x_0-x_1|)\int_{S_{u^\eps}\cap Q}  |[u^\eps](y)|\,\de\cH^{N-1}(y)+\eps \\
\leq&\, \tilde\omega_H(|x_0-x_1|)\bigg(1+\frac1{c_W^2}+\frac1{c_W}\int_Q W(x_1,\nabla u^\eps(y))\,\de y\\
&\, \phantom{\tilde\omega_H(|x_0-x_1|)\bigg(}+\frac1{c_\psi}\int_{S_{u^e}\cap Q} \psi(x_1,[u^\eps](y),\nu_{u^\eps}(y))\,\de\cH^{N-1}(y)\bigg)+\eps \\
\leq&\, \tilde\omega_H(|x_0-x_1|)\max\Big\{1+\frac1{c_W^2},\frac1{c_W},\frac1{c_\psi}\Big\} \big(1+E_{x_1}^{\bulk}(u^\eps;Q)\big)+\eps \\
\leq&\, \tilde\omega_H(|x_0-x_1|)\max\Big\{1+\frac1{c_W^2},\frac1{c_\psi}\Big\} \big(1+H(x_1,A,B)+\eps\big)+\eps \\
\leq&\, \tilde\omega_H(|x_0-x_1|)\max\Big\{1+\frac1{c_W^2},\frac1{c_\psi}\Big\} \big(\bar C_H(1+|A|+|B|^p)+\eps\big)+\eps,
\end{split}\]
where we have used, in sequence, \eqref{ineq1} in the first line, \eqref{pr1} and \eqref{pr5} in the second line, the coercivity \eqref{pr3} and \eqref{pr4} in the third line, \eqref{ineq1} again in the fifth line, and finally \eqref{ineq0} in the sixth line. By exchanging the roles of $x_0$ and $x_1$ and by the arbitrariness of $\eps>0$, we have obtained that $H^B$ satisfies \eqref{pr1} with $q=1$ and modulus of continuity $\omega_{H^B}(s)\coloneqq C(c_W,c_\psi,\bar C_H,|B|^p)\, \tilde\omega_H(s)$.

To show that $H^B$ satisfies property \eqref{(W1)_p} of Definition~\ref{gen_ass}, we recall that by \cite[Proposition~5.2]{CF1997} the relaxed bulk energy density is uniformly continuous in the $A$-variable. A closer inspection of the proof of that proposition, together with properties \eqref{W_cara} and \eqref{(psi4)} of Definition~\ref{gen_ass} to control the $x$-dependence, shows that for every $(x,B)\in\Omega\times\R{d\times N}$ the function $\R{d\times N}\ni A\mapsto H(x,A,B)$ is indeed Lipschitz continuous, so that 
\begin{equation}\label{LCA}
|H^B(x,A_1)-H^B(x,A_2)|\leq C|A_1-A_2|,
\end{equation}
which is \eqref{pr2} with $q=1$. Continuity with respect to $A$ follows. (We note, in passing, that \eqref{pr2} is trivially satisfied also with a general $q\geq1$.)

To show that $H^B$ satisfies property \eqref{W3} of Definition~\ref{gen_ass}, let $A_0\in\R{d\times N}$ be the matrix that makes the property hold true for $W$ and let $\{u_n\}\subset SBV(\Omega;\R{d})$ be the sequence provided by Theorem~\ref{appTHM} to approximate the structured deformation $(a_{A_0},B)$. Then, since $\{u_n\}\in\cC_{\seq}^{\bulk}(A_0,B)$, we apply Proposition~\ref{seq_char} and for $x\in\Omega$ we have
\[\begin{split}
H^B(x,A_0)=&\, \widetilde H(x,A_0,B)\leq \liminf_{n\to\infty}\bigg(\int_Q W(x,\nabla u_n(y))\,\de y+\int_{S_{u_n}\cap Q} \psi(x,[u_n](y),\nu_{u_n}(y))\,\de\cH^{N-1}(y)\bigg) \\
\leq&\, \liminf_{n\to\infty}\big(W(x,B)+C_\psi|D^s u_n|(Q)\big) \\
\leq&\, W(x,A_0)+C_W|A_0-B|\big(1+|A_0|^{p-1}+|B|^{p-1}\big)+C_\psi C\norma{(a_{A_0},B)}_{SD(Q)}, 
\end{split}\]
where we have used \eqref{pr2} with $q=p$ and \eqref{appEST}. This shows that $\norma{H^B(\cdot,A_0)}_{L^\infty(\Omega)}\leq C+\norma{W(\cdot,A_0)}_{L^\infty(\Omega)}$, so property \eqref{W3} of Definition~\ref{gen_ass} holds for $H^B$ with the same $A_0$ that was good for $W$.

Inequality \eqref{H coercivity} shows immediately that \eqref{pr3} holds for $H^B$ with $q=1$, and properties \eqref{W_cara} and \eqref{(W1)_p} of~$H^{B}$ imply
that $H^{B}$ is continuous.

For what concerns the surface energy density, we observe that, by its very definition \eqref{907}, $h$ satisfies properties \eqref{psi_0}, \eqref{(psi1)} with the constants $c_h\coloneqq c_\psi$ and $C_h\coloneqq C_\psi$, \eqref{(psi2)}, and \eqref{(psi3)} of Definition~\ref{gen_ass}.
In particular, in view of \eqref{psi_0}, it can be equivalently expressed as a function of $x$ and $\lambda \otimes \nu$, see (i) in Remark~\ref{rem 3.2}.
Therefore, as observed in \cite[Theorem 2.3]{S17}, the function $h$ can be seen as the restriction to rank-one matrices of a function $\Phi\colon\Omega \times\R{d \times N}\ni (x,M)\mapsto \Phi(x,M) \in [0, +\infty)$.

In particular \cite[equations (15), (16)]{S17} provide, for every $x \in \Omega$, 
\begin{align}
	\Phi(x, M) =\inf \left\{\sum\limits_{i=1}^{m}\psi (x,\lambda _{i},\nu _{i}):m\in 
	\mathbb{N}\backslash \{0\},(\lambda _{i},\nu _{i})\in \mathbb{R}^{d}\times 
	\mathbb{S}^{N-1}, 
	i =1,\cdots ,m, \, \sum\limits_{i=1}^{m}\lambda _{i}\otimes \nu
	_{i}=M\right\},  \label{alternate formula}
\end{align}%
together with the subadditivity and the positive $1$-homogeneity of the map $M\mapsto\Phi(x,M)$. Therefore, $\Phi(x,\cdot)$ is convex on $ \R{d\times N}$, for every $x\in\Omega$, and hence it is locally Lipschitzian and thus continuous, see \cite[Section 6.3, Theorem 1]{EG}.

To show that $h$ satisfies property \eqref{(psi4)} of Definition~\ref{gen_ass}, we argue in the following way. 
Let $x_0,x_1\in\Omega$ and $(\lambda,\nu)\in\R{d}\times\S{N-1}$ be given, and let $\eps>0$. 
By definition \eqref{907}, there exists $u^\eps\in\cC^{\surface}(\lambda,\nu)$ such that 
\begin{equation}\label{ineq2}
E_{x_1}^{\surface}(u^\eps; Q_\nu)< h(x_1,\lambda,\nu)+ \eps.
\end{equation} 
Therefore, since the same $u^\eps$ is also a competitor for $h(x_0,\lambda,\nu)$, we have, defining $\tilde\omega_h(s)\coloneqq c_\psi^{-1}\omega_\psi(s)$,
\[\begin{split}
h(x_0,\lambda,\nu)-&\,h(x_1,\lambda,\nu)\leq E_{x_0}^{\surface}(u^\eps;Q_\nu)-E_{x_1}^{\surface}(u^\eps;Q_\nu)+\eps \\
\leq&\, \omega_\psi(|x_0-x_1|)\int_{S_{u^\eps}\cap Q_\nu} |[u^\eps](y)|\,\de\cH^{N-1}(y)+\eps \\
\leq&\, \tilde\omega_h(|x_0-x_1|) \int_{S_{u^\eps}\cap Q_\nu} \psi(x_1,[u^\eps](y),\nu_{u^\eps}(y))\,\de\cH^{N-1}(y)+\eps= \tilde\omega_h(|x_0-x_1|) E_{x_1}^{\surface}(u^\eps;Q_\nu)+\eps \\
\leq&\, \tilde\omega_h(|x_0-x_1|)(h(x_1,\lambda,\nu)+\eps)+\eps \leq \tilde\omega_h(|x_0-x_1|)(C_\psi|\lambda|+\eps)+\eps, \\
\end{split}\]
where we have used, in sequence, \eqref{ineq2} in the first line, \eqref{pr5} in the second line, the coercivity \eqref{pr4} in the third line, and \eqref{ineq2} again  and \eqref{pr4} in the fourth line. By exchanging the roles of $x_0$ and $x_1$ and by the arbitrariness of $\eps>0$, we have obtained that $h$ satisfies \eqref{pr5} with  modulus of continuity $\omega_{h}(s)\coloneqq C_\psi\tilde\omega_h(s)$.
Continuity of $h$ in all of its variables follows. 

We prove now \eqref{H_B}.
Let $(x,A)\in\Omega\times\R{d\times N}$ be fixed, let $B_1,B_2\in\R{d\times N}$, and let $\{u_n\}\subset SBV(Q;\R{d})$ be a recovery sequence for $H(x,A,B_2)$. Let $\{v_n\}\subset SBV(Q;\R{d})$ be the sequence provided by Theorem~\ref{appTHM} to approximate $(0,B_1-B_2)$, so that the sequence $\{u_n+v_n\}\in\cC^{\bulk}_{\seq}(A,B_1)$.
Then we have, for $C_p$ a positive constant depending only on $p$,
\[\begin{split}
H(x,A,B_1)-&\,H(x,A,B_2)\leq \liminf_{n\to\infty} \big(E_{x}^{\bulk}(u_n+v_n;Q)-E_x^{\bulk}(u_n;Q)\big) \\
\leq&\, \liminf_{n\to\infty} \bigg(C_W|B_1-B_2|\int_Q (1+|\nabla u_n(y)+B_1-B_2|^{p-1}+|\nabla u_n(y)|^{p-1})\,\de y \\
&\,\phantom{\liminf_{n\to\infty} \bigg(} +\int_{S_{v_n}\cap Q} \psi(x,[v_n](y),\nu_{v_n}(y))\,\de\cH^{N-1}(y)\bigg) \\
\leq&\, \liminf_{n\to\infty} \bigg(C_W|B_1-B_2|\Big(1+ C_p(|B_1|^{p-1} +|B_2|^{p-1})+(1+C_p) \int_Q  |\nabla u_n(y)|^{p-1}\,\de y \Big)\\
&\, \phantom{\liminf_{n\to\infty} \bigg(}+ C_\psi\norma{D^s v_n}(Q)\bigg)\\
\leq&\, \liminf_{n\to\infty} \bigg( C_W(1+C_p)|B_1-B_2|\Big(1+|B_1|^{p-1}+|B_2|^{p-1}+\Big(\int_Q |\nabla u_n(y)|^p\,\de y\Big)^{\frac{p-1}p} \Big)\\
&\, \phantom{\liminf_{n\to\infty} \bigg(}+ C_\psi C_N|B_1-B_2|\bigg),\\
\end{split}\]
where we have used \eqref{pr2}, the sub-additivity property \eqref{(psi3)}, and the fact that $\nabla v_n=B_1-B_2$ in the second line, inequalities on powers and \eqref{pr4} in the third line, H\"{o}lder's inequality and the inequality in \eqref{817} in the fourth line. 
To treat the term containing $\nabla u_n$, we argue in the following way, letting $c_{W,p}\coloneqq \big(\max\{c_W^{-1},c_W^{-2}\}\big)^{\frac{p-1}p}$,
\[\begin{split}
\liminf_{n\to\infty}\bigg(\int_Q |\nabla u_n(y)|^p\,\de y\bigg)^{\frac{p-1}p}\leq&\, \liminf_{n\to\infty}\bigg(\frac1{c_W^2}+\frac1{c_W}\int_Q W(x,\nabla u_n(y))\,\de y\bigg)^{\frac{p-1}p} \\
\leq&\, \liminf_{n\to\infty} c_{W,p}\big(1+E_x^{\bulk}(u_n;Q)\big)^{\frac{p-1}p} = c_{W,p}(1+H(x,A,B_2))^{\frac{p-1}p} \\
\leq&\, c_{W,p}\big(1+\bar C_H(1+|A|+|B_2|^p)\big)^{\frac{p-1}p}\leq K \big(1+|A|^{\frac{p-1}p}+|B_2|^{p-1}\big)
\end{split}\]
where we have used the coercivity property \eqref{pr3}, the non-negativity of the surface energy, the fact that $\{u_n\}$ is a recovery sequence for $H(x,A,B_2)$, the upper bound in \eqref{ineq0}, and where we have called $K$ a constant depending on $p$, $c_{W,p}$, and $\bar C_H$.
Therefore we have, for a certain constant $C>0$,
$$H(x,A,B_1)-H(x,A,B_2)\leq C|B_1-B_2|(1+|B_1|^{p-1}+|B_2|^{p-1});$$
reversing the role of $B_1$ and $B_2$, we finally obtain \eqref{H_B}, namely that $H$ is $p$-Lipschitz continuous in the third variable.
As a consequence, $H$ is continuous in its arguments and the proof is concluded.
\end{proof}
%

\begin{remark}\label{rmkCF}
We observe the following facts.
\begin{itemize}
\item[(i)] We point out that the inequality from above in \eqref{ineq0} could have been equivalently deduced by the fact that $H$ satisfies \eqref{pr1}, \eqref{pr2}, and \eqref{H_B}.
\item[(ii)] We observe that \eqref{907} is $BV$-elliptic (see \cite[Definition~5.17]{AFP2000}): indeed, since $x\in\Omega$ is fixed, it arises as a formula identical to the homogeneous version of \cite[equation (11) and Proposition 2.2]{BDV1996} and \cite[equation (3.5), where the functional setting coincides with the current one]{AMMZ}. In light of this coincidence, the density~$h$ guarantees the lower semicontinuity of surface integrals also with respect to the $SBV_p$ convergence (see \cite{A1989} and \cite{AFP2000}). Roughly speaking one could say that $h$ is the $BV$-elliptic envelope of $\psi$; besides, the notion has been introduced in the purely brittle fracture context, \emph{i.e.} in the $SBV_p$ topology, (see \cite{BFLM,BF}). In particular, this observation goes in the direction that in the next relaxation iterations the $h$ remains unchanged.   
\item[(iii)] The above result still holds under weaker continuity conditions on the original bulk density $W$, thus enabling us to obtain a different, more analitically involved, energetic treatment of hierarchical (first-order) structured deformations. This will be the subject of a forthcoming paper.
\color{blue}
\end{itemize}
\end{remark}

\section{Hierarchical first-order structured deformations}\label{sec_hierarchies}
In this section we extend the Approximation Theorem~\ref{appTHM} and the integral representation Theorem~\ref{thm_CF} to hierarchical first-order structured deformations.
In order to achieve the approximation result, Theorem~\ref{appTHMh} below, an appropriate notion of convergence of a (multi-indexed) sequence of $SBV$ functions to an $(L+1)$-level (first-order) structured deformation $(g,G_1,\ldots,G_L)$ is introduced in Definition~\ref{S000}.
Then a recursive relaxation process is presented, in which Theorem~\ref{thm_CF} is recursively applied: in order to make this possible, Theorem~\ref{thm_propdens} guarantees that after each relaxation process, the densities obtained by freezing the microscopic gradient are still in the class $\ED(1)$, so that Theorem~\ref{thm_CF} can be applied one more time.

\begin{definition}
For $L\in\N{}$ and $\Omega\subset\R{N}$ a bounded connected open set, we define
$$HSD_L(\Omega)\coloneqq SBV(\Omega;\R{d})\times \underbrace{L^1(\Omega;\R{d\times N})\times\cdots\times L^1(\Omega;\R{d\times N})}_{L\text{-times}}$$
the set of \emph{$(L+1)$-level (first-order) structured deformations} on $\Omega$. 
For $p>1$, we define 
$$HSD_L^p(\Omega)\coloneqq SBV(\Omega;\R{d})\times \underbrace{L^p(\Omega;\R{d\times N})\times\cdots\times L^p(\Omega;\R{d\times N})}_{L\text{-times}}.$$
\end{definition}
An example of a three-level  structured deformation is provided by a stack of bundled papers (see \cite[Fig.~1]{DO2019}).
Deformations at the macrolevel (level 1) correspond to changes in shape of the entire stack, deformations at the first submacroscopic level (level 2) correspond to changes in shape of each of the bundles, and deformations at the second submacroscopic level (level 3) correspond to changes in shape of each single sheet of paper.
In this situation we have two levels of disarrangements,
the first  between bundled papers and the second between  sheets of paper.
Thus,  one can consider a macroscopic deformation field $g$ at each time, but also two
tensor fields $G_1$ and $G_2$ that provide the effects at the macrolevel of
geometrical changes without disarrangements at the submacroscopic levels 2 and 3, respectively, while the difference $\nabla g-G_{1}$ captures disarrangements in the form of slips or separations between adjacent bundles of papers, and the difference $G_{1}-G_{2}$ captures disarrangements in the form of slips or separations between individual sheets of papers.

\subsection{Approximation theorem for hierarchical (first-order) structured deformations}
The first result that we present is the definition of convergence of a sequence of $SBV$ functions to an $(L+1)$-level structured deformation $(g,G_1,\ldots,G_L)$ belonging to either $HSD_L(\Omega)$ or $HSD_L^p(\Omega)$. 
\begin{definition}\label{S000}
Let $L\in\N{}$, let $p>1$, let $(g,G_1,\ldots,G_L)\in HSD_L^{p}(\Omega)$, and let $\N{L}\ni(n_1,\ldots,n_L)\mapsto u_{n_1,\ldots,n_L}\in SBV(\Omega;\R{d})$ be a (multi-indexed) sequence. We say that $\{u_{n_1,\ldots,n_L}\}$ converges in the sense of $HSD_L^p(\Omega)$ to $(g,G_1,\ldots,G_L)$ if
\begin{itemize}
\item[(i)] $\displaystyle\lim_{n_1\to\infty}\cdots\lim_{n_L\to\infty} u_{n_1,\ldots,n_L} = g$, with each of the iterated limits in the sense of $L^1(\Omega;\R{d})$;
\item[(ii)] for all $\ell=1,\ldots,L-1$, 
$\displaystyle \lim_{n_{\ell+1}\to\infty}\cdots\lim_{n_L\to\infty} u_{n_1,\ldots,n_L} \eqqcolon g_{n_1,\ldots,n_\ell} \in SBV(\Omega;\R{d})$ and 
$$\lim_{n_1\to\infty}\cdots\lim_{n_{\ell}\to\infty} \nabla g_{n_1,\ldots,n_\ell}=G_{\ell},$$ 
with each of the iterated limits in the sense of weak convergence in  $L^p(\Omega;\R{d\times N})$;
\item[(iii)] $\displaystyle \lim_{n_1\to\infty}\cdots\lim_{n_L\to\infty} \nabla u_{n_1,\ldots,n_L} = G_L$ with each of the iterated limits in the sense of weak convergence in  $L^p(\Omega;\R{d\times N})$;
\end{itemize}
we use the notation $u_{n_1,\ldots,n_L}\wHSDKp(g,G_1,\ldots,G_L)$ to indicate this convergence.

If $(g,G_1,\ldots,G_L)\in HSD_L(\Omega)$ and if the weak $L^p$ convergences above are replaced by weak-* convergences in $\cM(\Omega;\R{d\times N})$, then we say that $\{u_{n_1,\ldots,n_L}\}$ converges in the sense of $HSD_L(\Omega)$ to $(g,G_1,\ldots,G_L)$ and we use the notation $u_{n_1,\ldots,n_L}\wHSDKstar(g,G_1,\ldots,G_L)$ to indicate this convergence.
\end{definition}

The sequential application of the idea behind the Approximation Theorem~\ref{appTHM} provides the method for constructing a (multi-indexed) sequence $\{u_{n_1,\ldots,n_L}\}$ that approximates an $(L+1)$-level structured deformation $(g,G_1,\ldots,G_L)$.
\begin{theorem}[Approximation Theorem for $(L+1)$-level structured deformations]\label{appTHMh}
Let  $(g,G_1,\ldots,G_L)$ belong to either $HSD_L(\Omega)$ or $HSD_L^p(\Omega)$. 
Then there exists a sequence $(n_1,\ldots,n_L)\mapsto u_{n_1,\ldots,n_L}\in SBV(\Omega;\R{d})$ 
converging to $(g,G_1,\ldots,G_L)$ in both of the senses of Definition~\ref{S000}.
\end{theorem}
\begin{proof}
We set $G_0\coloneqq\nabla g$, for convenience.
For every $\ell=1,\ldots,L$, let $u_\ell\in SBV(\Omega;\R{d})$ be the functions provided by Theorem~\ref{Al} such that $\nabla u_\ell=G_{\ell-1}-G_\ell$ and let $n_\ell\mapsto \bar u_{n_\ell}$ be the piecewise constant sequence approximating $u_\ell$ in $L^1(\Omega;\R{d})$ provided by Lemma \ref{ctap}.
We claim that the sequence of functions
\begin{equation}\label{S002}
u_{n_1,\ldots,n_L}\coloneqq g+\sum_{\ell=1}^L (\bar u_{n_\ell}-u_\ell),
\end{equation}
approximates $(g, G_1,\ldots,G_L)\in HSD_L(\Omega)$ in the sense of the convergence $\wHSDKstar$ or $(g, G_1,\ldots,G_L)\in HSD_L^p(\Omega)$ in the sense of the convergence $\wHSDKp$, see Definition~\ref{S000}.
Indeed,
\begin{equation}\label{T001}
\begin{split}
& \lim_{n_1\to\infty}\cdots\lim_{n_L\to\infty} u_{n_1,\ldots,n_L}=  \lim_{n_1\to\infty}\cdots\lim_{n_L\to\infty} \bigg( g+\sum_{\ell=1}^L (\bar u_{n_\ell}-u_\ell)\bigg) \\
=& \lim_{n_1\to\infty}\cdots\lim_{n_{L-1}\to\infty} \bigg( g+\sum_{\ell=1}^{L-1} (\bar u_{n_\ell}-u_\ell)\bigg)=\cdots=g, \qquad\text{in $L^1(\Omega;\R{d})$},
\end{split}
\end{equation}
proving (i). 
Using \eqref{T001}, we have that 
\begin{equation*}
\begin{split}
g_{n_1,\ldots,n_\ell} &\coloneqq \lim_{n_{\ell+1}\to\infty}\cdots\lim_{n_L\to\infty} u_{n_1,\ldots,n_L} =  \lim_{n_{\ell+1}\to\infty}\cdots\lim_{n_L\to\infty} \bigg( g+\sum_{j=1}^L (\bar u_{n_j}-u_j)\bigg) \\
& = g+\sum_{j=1}^\ell (\bar u_{n_j}-u_j) \in SBV(\Omega;\R{d}),
\end{split}
\end{equation*}
from which, observing that 
\begin{equation}\label{T001c}
\nabla g_{n_1,\ldots,n_\ell}= \nabla \bigg( g+\sum_{j=1}^\ell (\bar u_{n_j}-u_j) \bigg)=\bigg( \nabla g+\sum_{j=1}^\ell (G_j-G_{j-1}) \bigg)=G_0 +\sum_{j=1}^\ell (G_j-G_{j-1})=G_\ell,
\end{equation}
we immediately obtain that 
\begin{equation*}
\lim_{n_1\to\infty}\cdots\lim_{n_\ell\to\infty} \nabla g_{n_1,\ldots,n_\ell} = G_\ell,
\end{equation*}
both in the weak $L^p(\Omega;\R{d\times N})$ sense and in the weak-* $\cM(\Omega;\R{d\times N})$ sense, which is condition (ii) (equality \eqref{T001c} tells us that the sequence $\{\nabla g_{n_1,\ldots,n_\ell}\}$ is indeed a constant sequence). 
Finally, condition (iii) follows upon observing that 
\begin{equation}\label{T001d}
\nabla u_{n_1,\ldots,n_L}=\nabla g+\sum_{\ell=1}^L (G_\ell-G_{\ell-1})=G_L,
\end{equation}
so that 
$$\lim_{n_1\to\infty}\cdots\lim_{n_L\to\infty} \nabla u_{n_1,\ldots,n_L} = G_L,$$
both in the weak $L^p(\Omega;\R{d\times N})$ sense and in the weak-* $\cM(\Omega;\R{d\times N})$ sense (equality \eqref{T001d} tells us that the sequence $\{\nabla u_{n_1,\ldots,n_L}\}$ is indeed a constant sequence). 
The proof is concluded.
\end{proof}

\subsection{Recursive relaxation}\label{sec_RRP}

The assignment of an energy to an $(L+1)$-level structured deformation $(g,G_{1},\cdots ,G_{L})$ can be approached in a variety of ways.
The simple approach taken here is through a recursive scheme whose initial step assigns an energetic response $E_{L+1}$ to deformations $u_{L+1}$ at level $L+1$ through a bulk energy density $(x,A)\mapsto W_{0}(x,A)$ and an interfacial energy density $(x,\lambda ,\nu )\mapsto \psi _{0}(x,\lambda,\nu )$, with $(W_{0},\psi _{0})\in \ED(p)$, and then relaxes that energetic response to yield an energetic response $E_{L}$ to structured deformations $(\tilde{g}_{L},\tilde{G}_{L})$ at level~$L$. 
Here, the field  $\tilde{g}_{L}$ represents a deformation at level~$L$ and $\tilde{G}_{L}$ represents the amount of deformation at level~$L$ that is due to smooth deformations at level $L+1$, \emph{i.e.}, $\tilde{G}_{L}$ is the deformation at level~$L$ without $(L+1)$-level disarrangements. 
The energetic response $E_{L}$ to structured deformations $(\tilde{g}_{L},\tilde{G}_{L})$ is determined according to Theorem~\ref{thm_CF} by means of the cell formulas \eqref{rel_en_dens} for the relaxed bulk energy density~$W_{1}$ and the relaxed interfacial density~$\psi_{1}$ in terms of~$W_{0}$ and~$\psi _{0}$. 
Although the initial bulk energy density~$W_{0}$ has domain $\Omega \times\R{d\times N}$, the relaxed bulk density~$W_{1}$ has domain $\Omega \times\R{d\times N}\times\R{d\times N}$, so that the initial recursive step results in the introduction of an additional matrix variable for the relaxed bulk density that appears within the integrand in the expression for $E_{L}(\tilde{g}_{L},\tilde{G}_{L})$ as the argument~$\tilde{G}_{L}(x)$ of $W_{1}(x,\nabla \tilde{g}_{L}(x),\tilde{G}_{L}(x))$.
The context of Theorem~\ref{thm_CF} assures that the arguments of~$\psi _{1}$ are the same as those of~$\psi _{0}$.

The simplicity of the present approach lies in the freezing of the additional argument~$\tilde{G}_{L}(x)$ at a given but arbitrary matrix~$B_{L}$ for all subsequent steps which, accordingly, we describe as ``partial relaxations''. 
This allows us to replace $(x,A)\mapsto W_{0}(x,A)$ in the initial step by $(x,A)\mapsto W_{1}(x,A,B_{L})$ and to replace~$\psi_{0}$ by~$\psi _{1}$ to start the next recursive step. 
Theorem~\ref{thm_CF} then can be applied to $E_{L}$, with densities $(x,A)\mapsto W_{1}(x,A,B_{L})$ and~$\psi _{1}$, because Theorem~\ref{thm_propdens} guarantees that $(W_{1}(\cdot ,\cdot,B_{L}),\psi _{1})\in \ED(1)$, thereby accomplishing the first in the sequence of partial relaxations. 
We claim here that a single initial relaxation followed by~$L$ partial relaxations can be carried out, resulting in a pair $(W_{L}(\cdot ,\cdot ,B_{1},\cdots ,B_{L}),\psi _{L})$ of multiply
relaxed energy densities in terms of the~$L$ matrices $B_{1},\cdots ,B_{L}$ that were introduced, one-by-one, at successive steps during the recursive procedure. 
Consequently, the energy~$E_{1}$ at level~$1$, the macroscopic level, can be assigned to an $(L+1)$-level structured deformation $(g,G_{1},\cdots ,G_{L})$ via the formula
\begin{equation*}\label{level 1 energy}
E_{1}(g,G_{1},\cdots ,G_{L}) \coloneqq  \int_{\Omega} W_{L}(x,\nabla g(x),G_{1}(x),\cdots ,G_{L}(x))\,\de x 
+\int_{\Omega\cap S_g} \psi_{L}(x,[g](x),\nu_{g}(x))\,\de\cH^{N-1}(x).  
\end{equation*}%
This assignment of energy corresponds to a step-wise optimization as the number of levels covered by the hierarchical structured deformation is increased by one at each recursive relaxation step, while, as is detailed in Section~\ref{sec_CO}, other possible assignments of energy might optimize without explicit passage from one level to another and without freezing the values of any of the fields $G_{1},\cdots ,G_{L}$.

The discussion of recursive relaxation in the preceding paragraphs is preliminary to a recursive specification in which the index~$k$ for a recursive stage represents the level $L+1-k$ associated with submacroscopic refinement. 
For example, the recursive stage $k=0$ corresponds to the initial level $L+1$, the ``finest'' submacroscopic level, while the stage $k=L$ corresponds to level~$1$, the macroscopic level. 
In case the recursion should break down at some stage, we specify an integer $m\in \{1,\cdots,L\}$ to denote that stage and, accordingly, require that $k\in \{0,\cdots ,m\}$. 
Given $p>1$, $m\in \{0,\cdots ,L\}$, and $(W_{0},\psi _{0})\in \ED(p)$, we specify recursively pairs $(W_{k},\psi _{k})$ for $k\in \{1,\cdots ,m\}$ with
\begin{equation}\label{kth relaxed densities}
W_{k}\colon\Omega \times\big(\R{d\times N}\big)^{k+1}\to[0,+\infty)\quad\text{and}\quad \psi_{k}\colon\Omega \times \R{d}\times\S{N-1}\to[0,+\infty )
\end{equation}%
emerging from the following counterparts of the cell formulae \eqref{rel_en_dens} in which, for notational convenience, we write $\bB_{k}\coloneqq(B_{k},\ldots,B_{1})\in (\R{d\times N})^k$ for $k\in\{1,\ldots,m\}$.
Hence the cell formulas at stage $k$ are, for all $A_k,B_{1},\ldots ,B_{k}\in\R{d\times N}$,%
\begin{equation}\label{Wkpsik}
\begin{split}
W_{k}(x,A_{k},\mathbf{B}_{k})\coloneqq&\, \inf \bigg\{\int_Q W_{k-1}(x,\nabla u(y),\mathbf{B}_{k-1})\,\de y \\
&\, \phantom{\inf\bigg\{}+ \int_{Q\cap S_{u}}\psi _{k-1}(x, [u](y),\nu_{u}(y))\,\de\cH^{N-1}(y): u\in \cC^{\bulk}(A_k,B_k)\bigg\}, \\
\psi _{k}(x,\lambda ,\nu) \coloneqq&\, \inf \bigg\{\int_{Q_{\nu }\cap S_{u}}\psi_{k-1}(x,[u](y),\nu_{u}(y))\,\de\cH^{N-1}(y):u\in \cC^\surface(\lambda,\nu)\bigg\}.
\end{split}
\end{equation}
We note that in the passage from the bulk energy $W_{k-1}$ to $W_{k}$, the expression $W_{k-1}(x, A_{k-1},B_{k-1},\ldots ,B_{1})$ is replaced by $W_{k}(x, A_{k},B_{k},B_{k-1},\ldots ,B_{1})$, so that the single matrix $A_{k-1} $ in the arguments of $W_{k-1}$ is replaced by the pair of matrices $(A_{k},B_{k})$ in the arguments of $W_{k}$.

In particular, defining 
\begin{equation}\label{W_K^B}
\Omega\times\R{d\times N}\ni(x,A)\mapsto W_{k}^{\bB_k}(x,A)\coloneqq W_k(x,A,\bB_{k}),
\end{equation}
by Theorem~\ref{thm_propdens} $(W_k^{\bB_k},\psi_k)\in\ED(1)$ for every $k\in\{0,\ldots,m\}$.
Due to this fact, there is no stopping point to this process which allows us to consider energies associated to hierarchies of structured deformations with as many levels as needed.

For each $(L+1)$-level structured deformation $(g,G_{1},\cdots ,G_{L})$, for each level $\ell\in \{1,\cdots,L\}$, and for each $\tilde{g}_{\ell }\in SBV(\Omega ,\R{d})$, the $(L+2-\ell)$-tuple $(\tilde{g}_{\ell },G_{\ell},\cdots ,G_{L})$ is an $(L+2-\ell)$-level structured deformation to which we assign the (partially) relaxed energy
\begin{equation}\label{EHSDk}
\begin{split}
E_{\ell}(\tilde g_\ell, G_\ell,\dots, G_{L})\coloneqq& \int_\Omega W_{L+1-\ell}(x,\nabla \tilde g_\ell(x), G_{\ell}(x),\dots, G_L(x))\,\de x \\
&\,+\int_{\Omega \cap S_g}  \psi_{L+1-\ell}(x, [\tilde g_\ell](x),\nu_{\tilde g_\ell}(x))\,\de\cH^{N-1}(x).
\end{split}
\end{equation}
We observe that the arguments given in the proof of Theorem \ref{thm_propdens} to obtain property \eqref{W_cara} for the function $x\mapsto H(x, A,B)$, the Lipschitz continuity of $A\mapsto H(x,A, B)$ (see \eqref{LCA}), and the $p$-Lipschitz continuity of $B\mapsto H(x, A,B)$ (see \eqref{H_B}), applied to $W_{L+1-\ell}$,  guarantee the measurability and the integrability of the integrand $x\mapsto W_{L+1-\ell}(x,\nabla \tilde{g}_{\ell }(x),G_{\ell }(x),\cdots ,G_{L}(x))$, so that the energy \eqref{EHSDk} is well defined.

The reasoning that we have just carried out provides a proof of the following statement.
\begin{theorem}[Well-posedness of the recursive relaxation]\label{thm_RR}
Let $p>1$ and let $(W_0,\psi_0)\in\ED(p)$; let $L\in\N{}$ and $(g,G_1,\ldots,G_L)\in HSD_L^p(\Omega)$. 
Then for every level $\ell\in\{1,\cdots,L\}$ and $\tilde{g}_{\ell}\in SBV(\Omega;\R{d})$ the (partially) relaxed energy $E_{\ell }(\tilde{g}_{\ell },G_{\ell },\cdots,G_{L})$ at level $\ell $ in \eqref{EHSDk} is well defined.
\end{theorem}

\subsection{An explicit example}\label{sec_EX}
We present here an example of the recursive process described in Section~\ref{sec_RRP} that, although not strictly covered by the hypotheses of Theorem~\ref{thm_RR}, is also well-posed in the sense of Theorem~\ref{thm_RR} and yields explicit formulas for the relaxed energy at every stage of the recursion. 
We set $d=N$, let $p>1$, and consider initial energy densities $W_0\colon\R{N\times N}\to[0,+\infty)$ convex and satisfying property \eqref{(W1)_p} 
of Definition~\ref{gen_ass} with $q=p$ and $\psi_0\colon\R{N}\times\S{N-1}\to[0,+\infty)$ given by
\begin{equation}\label{example interfacial energy}
\psi_0(\lambda ,\nu )=\left\vert \lambda \cdot \nu \right\vert .
\end{equation}%
Because $p>1$ and because $\psi _{0}$ satisfies \eqref{(psi2)} and \eqref{(psi3)} in Definition~\ref{gen_ass}, the conclusions of Theorem~\ref{thm_CF}, including the cell formulas \eqref{rel_en_dens}, apply in this example (see Remark~\ref{rem 3.2}). 
It is proved in \cite[Proposition~3.6]{MMO} that under these hypotheses we have the following form of the relaxed bulk energy density in \eqref{906}, which we will denote here by $W_1$,
\begin{equation}\label{3.13}
W_1(A_1,B_1)=W_0(B_1)+\inf\bigg\{\int_{Q\cap S_u} \psi_0([u](y),\nu_u(y))\,\de\cH^{N-1}(y): u\in \cC^{\bulk}(A_1,B_1)\bigg\}.
\end{equation}
The minimization problem in the right-hand side of formula \eqref{3.13} was solved in \cite{nogap,OP2015}, where it was proved that 
$$\inf\bigg\{\int_{Q\cap S_u} \psi_0([u](y),\nu_u(y))\,\de\cH^{N-1}(y): u\in \cC^{\bulk}(A_1,B_1)\bigg\}=|\tr(A_1-B_1)|$$
(see \cite[Theorem~2.3]{S17} for a general formula).
The results in \cite{OP2015,S17} also give the explicit expression for the relaxed surface energy density in \eqref{907}, which we will denote here by $\psi_1$, 
$$\psi_1(\lambda,\nu)=\psi_0(\lambda,\nu)=|\lambda\cdot\nu|.$$
(It is worth to underline that $\psi_0$ is jointly convex, hence $BV$-elliptic (see \cite[Definition~5.17 and Theorem~5.20]{AFP2000} and Remark~\ref{rmkCF}(ii)), hence $\psi_1=h= \psi_0$.)
Therefore we have
\begin{equation}\label{example densities 1}
\begin{split}
W_{1}(A_{1},B_{1}) =&\,W_0(B_{1})+\vert \tr(A_{1}-B_{1})\vert \\
\psi _{1}(\lambda ,\nu ) =&\,\vert \lambda \cdot \nu \vert. 
\end{split}
\end{equation}
Upon noticing that the function $\R{N\times N}\ni A\mapsto W_1^{B_1}(A)\coloneqq W_{1}(A ,B_{1})$ is also convex and satisfies \eqref{(W1)_p} in Definition~\ref{gen_ass} with $q=p$, we may then perform a second relaxation via the argument that led to \eqref{example densities 1} to obtain the relaxed densities 
\begin{equation}\label{example densities 2}
\begin{split}
W_{2}(A_{2},B_{2,}B_{1}) =&\,W_{1}^{B_1}(B_{2})+\vert\tr(A_{2}-B_{2})\vert =W_1(B_2,B_1)+|\tr(A_2-B_2)| \\
&\, =W_0(B_{1})+\vert\tr(B_{2}-B_{1})\vert +\vert\tr(A_{2}-B_{2})\vert  \\
\psi _{2}(\lambda ,\nu) =&\,\vert \lambda \cdot \nu \vert.
\end{split}
\end{equation}%
It is straightforward to show by induction that the recursive relaxation determined by $W_0$ and $\psi_0$ in this example 
admits relaxed densities at the $k^{\text{th}}$ recursive stage given by%
\begin{subequations}\label{example_k}
\begin{eqnarray}
W_{k}(A_{k},\bB_{k}) &\!\!\!\!=&\!\!\!\! \vert \tr(A_{k}-B_{k})\vert+\sum_{j=2}^{k}\vert \tr(B_{j}-B_{j-1})\vert +W_0(B_{1}) \label{example bulk k} \\
\psi _{k}(\lambda ,\nu ) &\!\!\!\!=&\!\!\!\! \vert \lambda \cdot \nu\vert. \label{example interfacial k}
\end{eqnarray}%
\end{subequations}
In this example, the $L$-fold relaxed bulk and interfacial energy densities $W_{L}$ and $\psi_{L}$ just obtained can be inserted in \eqref{EHSDk} to assign the energy $E_{1}$ to a hierarchical structured deformation $(g,G_{1},\ldots,G_{L})$
\begin{equation}\label{EK energy}
E_{1}(g,G_{1},\ldots G_{L}) \coloneqq \int_{\Omega} W_{L}(\nabla g(x),G_{1}(x),\ldots ,G_{L}(x)) \,\de x +\int_{\Omega \cap S_{g}} \vert [g](x)\cdot \nu_{g}(x)\vert \,\de\cH^{N-1}(x).  
\end{equation}%
If we put $G_{0}\coloneqq \nabla g$ and use \eqref{example bulk k} with $k=L$, the volume integrand  $W_{L}(\nabla g,G_{1},\ldots ,G_{L})$ in \eqref{EK energy} can be written as
\begin{equation}\label{bulk energy integrand}
\sum_{\ell=1}^{L}\vert \tr(G_{\ell}-G_{\ell-1})\vert +W_0(G_{L}).  
\end{equation}%
The matrix-valued fields $G_{\ell-1}-G_{\ell}$ in \eqref{bulk energy integrand} are the contributions to the macroscopic deformation gradient~$\nabla g$ from the disarrangements arising in upscaling from  submacroscopic level~$\ell$ to submacroscopic level~$\ell-1$, so that the sum over~$\ell$ represents a bulk energy density arising from hierarchical disarrangements.  
The remaining term $W_0(G_{L})$ represents a bulk energy density arising from the contributions at the macrolevel of the deformation without disarrangements at submacroscopic level~$L$.

\section{Conclusions and outlook}\label{sec_CO}
In this work we have taken the first steps towards a variational theory for hierarchical  structured deformations, introduced by Deseri and Owen in \cite{DO2019}.
In the spirit of the variational framework introduced by Choksi and Fonseca in \cite{CF1997}, we proved two main results: (i) the Approximation Theorem~\ref{appTHMh}, stating that each $(L+1)$-level structured deformation $(g,G_1,\ldots,G_L)$ can be approximated by a (multi-indexed) sequence $\{u_{n_1,\ldots,n_L}\}\subset SBV(\Omega;\R{d})$ in the precise sense introduced in Definition~\ref{S000}; (ii) the recursive relaxation procedure described in Section~\ref{sec_RRP} and culminating in Theorem~\ref{thm_RR}, stating that the energy assigned to an $(L+1)$-level structured deformation $(g,G_1,\ldots,G_L)$ by means of iterated applications of the integral representation Theorem~\ref{thm_CF} is indeed well defined; this second result rests on Theorem~\ref{thm_propdens}, the main technical effort of this paper, in which we prove that the relaxed bulk and surface energy densities still belong to the class of energies that can be relaxed. It is important to notice that these findings provide a validation of the proposal initiated in \cite{DO2019}, motivating the introduction of the notion of hierarchical structured deformations to model materials exhibiting multiple submacroscopic levels.

The Approximation Theorem~\ref{appTHM} for first-order structured deformations $(g,G)$ is crucial for the relaxation process \eqref{910}, since it implies that the class $\cR_p(g,G)$ defined in \eqref{R_p} is not empty, so that \eqref{910} is a meaningful problem. The integral representation Theorem~\ref{thm_CF} then shows that starting from an initial energy $E$ as in \eqref{103}, the relaxed energy $I(g,G)$ assigned to $(g,G)$ retains the integral form, and provides the cell formulae \eqref{rel_en_dens} for the relaxed bulk and surface energy densities $H$ and $h$, respectively. 

At the moment, Theorems~\ref{appTHMh} and~\ref{thm_RR} work independently: the way we assign the energy to the $(L+1)$-level structured deformation $(g,G_1,\ldots,G_L)$ is independent of the fact that $(g,G_1,\ldots,G_L)$ can be approximated by means of the sequence $\{u_{n_1,\ldots,n_L}\}$. Thanks to Theorem~\ref{thm_propdens}, our proposal for defining $E_1(g,G_1,\ldots,G_L)$ via \eqref{EHSDk} (for $\ell=1$) is meaningful, as formula \eqref{EK energy} shows.

It would be amenable, and extremely interesting from the mathematical point of view, to define the energy of an $(L+1)$-level structured deformation $(g,G_1,\ldots,G_L)\in HSD_L(\Omega)$ via relaxation of the initial energy $E$ of \eqref{103} directly, namely by
\begin{equation}\label{E_K rel}
I_1(g,G_1,\ldots,G_L)\coloneqq\inf\Big\{\liminf_{n\to\infty} E(u_{n_1,\ldots,n_L}): \{u_{n_1,\ldots,n_L}\}\in\cR_p(g,G_1,\ldots,G_L)\Big\},
\end{equation}
where a possible definition of the class $\cR_p$ in this context could be
\begin{equation}\label{R_p h}
\begin{split}
\cR_p(g,G_1,\ldots,G_L)\coloneqq\Big\{\{u_{n_1,\ldots,n_L}\}\subset SBV(\Omega;\R{d}): &\, u_{n_1,\ldots,n_L}\wHSDKstar(g,G_1,\ldots,G_L), \\
&\,\sup_{n_1,\ldots,n_L\in\N{L}}\norma{\nabla u_{n_1,\ldots,n_L}}_{L^p(\Omega;\R{d\times N})}<+\infty\Big\}.
\end{split}
\end{equation}
The Approximation Theorem~\ref{appTHMh} grants that the class $\cR_p$ defined in \eqref{R_p h} is non empty, so that, also in the context of hierarchies of structured deformations, a definition by relaxation, such as \eqref{E_K rel} is meaningful. The properties of such a definition are object of research in progress.
Whether the energy $I_1$ admits an integral representation and in which fashion the energy densities depend on the fields $\nabla g$ and $G_1,\ldots,G_L$, and on the jump $[g]$ is still unknown at the moment, as is the relationship between the energies $E_1$ of \eqref{EHSDk} (for $\ell=1$) and $I_1$. It seems not unreasonable to expect the latter to be smaller than the former: in the relaxation process defined in \eqref{E_K rel}, the energy $I_1$ is given as the most energetically convenient way of reaching $(g,G_1,\ldots,G_L)$, whereas in \eqref{EHSDk}, the energy $E_1$ is the outcome of the iteration of the relaxation process \eqref{910}, in which only one level at a time is upscaled, keeping the others frozen. This is evident in the way we define the bulk energy densities $W_k$ at each stage by \eqref{W_K^B}, by keeping the $k$ fields $\bB_k=(B_k,\ldots,B_1)$ fixed, thus introducing some constraints.
On the other hand, as pointed out in Remark~\ref{rmkCF}(ii), the properties listed in Theorem~\ref{thm_propdens} remain valid under weaker continuity assumptions on the original bulk density. The study of the energy $I_1$ is undertaken under these weaker conditions and is the subject of a forthcoming paper.

\subsection*{Acknowledgments} The authors warmly thank the hospitality of the departments of mathematics of Carnegie Mellon University and of Politecnico di Torino, where part of this research was developed.
MM and EZ are members of the Gruppo Nazionale per l'Analisi Matematica, la Probabilit\`a e le loro Applicazioni (GNAMPA) of the Istituto Nazionale di Alta Matematica ``F.~Severi'' (INdAM). They also acknowledge partial funding from the GNAMPA Project 2019 \emph{Analysis and optimisation of thin structures}.

The research of ACB was partially supported by National Funding from FCT--Funda\c{c}\~{a}o para a Ci\^{e}ncia e a Tecnologia, through project UIDB/04561/2020.
The research of JM is partially supported by FCT/Portugal through CAMGSD, IST-ID, projects UIDB/04459/2020 and UIDP/04459/2020.
The work of MM is partially supported by the MIUR project \emph{Dipartimenti di Eccellenza 2018-2022} (CUP E11G18000350001) and by the \emph{Starting grant per giovani ricercatori} of Politecnico di Torino.
EZ is indebted to IST-Universidade de Lisboa for their kind support and hospitality. She also acknowledges  GNAMPA support for travel expenses in 2022.


\bibliographystyle{plain}

\end{document}